\documentclass{amsart}

\usepackage{latexsym,amssymb,amsmath, amsfonts}
\usepackage{amsthm,amsrefs}
\usepackage{amscd}

\newcommand{\R}{\mathbb{R}}
\newcommand{\C}{\mathbb{C}}
\newcommand{\Z}{\mathbb{Z}}

\renewcommand{\le}{\leqslant}
\renewcommand{\ge}{\geqslant}
\renewcommand{\a}{\alpha}

\newcommand{\be}{\begin{equation}}
\newcommand{\en}{\end{equation}}
\newcommand{\ee}{\end{equation}}

\newcommand{\bt}{\begin{theorem}}
\newcommand{\et}{\end{theorem}}
\newcommand{\bp}{\begin{proof}}
\newcommand{\ep}{\end{proof}}
\newcommand{\bc}{\begin{cor}}
\newcommand{\ec}{\end{cor}}
\newcommand{\bl}{\begin{lemma}}
\newcommand{\el}{\end{lemma}}
\newcommand{\bprop}{\begin{prop}}
\newcommand{\eprop}{\end{prop}}

\renewcommand{\Re}{\mathrm{Re}}

\newcommand{\N}{\mathbb{N}}

\newcommand{\ls}{\lesssim}
\newcommand{\gs}{\gtrsim}

\newtheorem{theorem}{Theorem}[section]

\newtheorem{lemma}{Lemma}[section]

\newtheorem{prop}{Proposition}[section]

\newtheorem{cor}{Corollary}[section]
\newtheorem*{theoremA}{Theorem A}
\newtheorem*{theoremB}{Theorem B}
\newtheorem*{theoremC}{Theorem C}

\numberwithin{theorem}{section} \numberwithin{definition}{section}

\author{Neal Bez}
\address{Graduate School of Mathematical Sciences, The University of Tokyo, 3-8-1 Komaba, Meguro-ku, 153-8914 Tokyo, Japan}\email{bez@ms.u-tokyo.ac.jp}

\author{Keith M. Rogers}
\address{Instituto de Ciencias Matemáticas CSIC-UAM-UC3M-UCM, 28049 Madrid, Spain}\email{keith.rogers@icmat.es}

\author{Shunya Toyoshima}
\address{Department of Mathematics, Graduate School of Science and Engineering, Saitama University, Saitama 338-8570, Japan}\email{s.toyoshima.845@ms.saitama-u.ac.jp}

\thanks{Supported by JSPS Kakenhi grants 22H00098, 23K25777 and 25K24627 (Bez), and MICINNU grants CEX2023-001347-S, PID2021-124195NB-C33 and PID2024-158664NB-C22 (Rogers)}

\date{}
\title[{Orthonormal Sobolev estimates with fractal measures}]{Orthonormal Sobolev estimates\\ with fractal measures}

\begin{document}

\begin{abstract} We prove a fractal version of Lieb's Hardy--Littlewood--Sobolev inequality for orthonormal functions. On the one hand, this can be viewed as a trace theorem for orthonormal functions.
 On the other, it allows us to recover the Rozenblum--Tashchiyan bound for the number of negative eigenvalues  of $-\Delta-\mu$, where $\mu$ is a shell potential. We also  recover Rozenblum's bound  for the sum of negative eigenvalues via a  Lieb--Thirring kinetic inequality. Our proof is direct, avoiding both Schatten classes and variational arguments.
We first reprove Adams' fractal Hardy--Littlewood--Sobolev inequality (for single functions) via Fourier analysis. This yields the required endpoint estimate as well as a bound for the interaction energy of Frostman measures.
\end{abstract}

\maketitle

\section{Introduction}\label{intro}

For $s<d/2$, where $d\in\N$ is the ambient dimension, we consider fractional derivatives formally defined via the Fourier transform by $$D^{-s}f:=\big(|\cdot|^{-s}\widehat{f}\,\big)^\vee.$$
In particular, we consider their integrability with respect to Frostman measures. Throughout $\mu$ will denote a nonnegative Borel measure that satisfies $$
[\mu]_\a:=\sup_{x\in\mathbb{R}^d,\,r>0}\frac{\mu\big(B(x,r)\big)}{r^{\alpha}}<\infty,\qquad
0<\alpha\leq d,
$$
where $B(x,r)$ denotes the Euclidean ball, centred at $x$, of radius $r>0$. 

In the fourth section, we will reprove the fractal Hardy--Littlewood--Sobolev inequality,  due to Adams \cite{A1}.

\begin{theoremA}\label{A} Let $2< q<\infty$ and $s=d/2-\alpha/q$. Then there is a constant $C>0$ such that\begin{align*}
\|D^{-s}f\|_{L^q(d\mu)}
\le C[\mu]_\a^{1/q}\,\|f\|_{L^2(\R^d)}
\end{align*}
whenever $f\in L^2(\R^d)$.
\end{theoremA}

In the fifth section, we extend this to orthonormal functions. The $p=2$ case follows directly from Theorem A, by the triangle inequality. The original orthonormal version, mapping from $\ell^{q/2,1}(\C)$ with $\mu$ Lebesgue measure, is due to Lieb~\cite{L}. 

\begin{theorem}\label{1.1} Let $2\le p< q<\infty$ and $s=d/2-\alpha/q$. Then there is a constant $C>0$ such that
\begin{align*}
\bigg\|\sum_{j} \gamma_j|D^{-s}f_j|^2\bigg\|_{L^{q/2}(d\mu)}
\le  C[\mu]_\a^{2/q}\,\|\gamma\|_{\ell^{p/2}}
\end{align*}
whenever $(f_j)_j$ is orthonormal in $L^2(\R^d)$ and $\gamma=(\gamma_j)_j\in\ell^{p/2}(\C)$.
\end{theorem}

 For this we prove a more general estimate, in which we treat derivatives on both sides of the inequality.   When~$\mu$ is defined on a lower-dimensional set $\Gamma\subset \R^d$, another consequence is the following trace inequality for orthonormal functions. For single function estimates of this type, with $s_1,s_2\in \Z$,  see \cite[Section 1.4.7]{M}.

\begin{theorem}\label{1.2} Let $0<\theta=\frac{d-2s_1}{\alpha+2s_2}<1$. Then there is a constant $C>0$ such that
$$
\sum_{j=1}^N \int_{\Gamma}  |D^{-s_1}f_j(x)|^2 d\mu(x) \le C[\mu]_\a^{\theta}\mu(\Gamma)^{1-\theta}\bigg(\sum_{j=1}^N \|D^{s_2}f_j\|^2_{L^2(\R^d)}\bigg)^{\theta}
$$
whenever $(f_j)_{j=1,\ldots,N}$ is orthonormal in $L^2(\R^d)$ and $N\ge1$.
\end{theorem}

Taking $s_1=0$ and $s_2=1$, the right-hand side can be rewritten in terms of gradients and there is no longer any mention of the Fourier transform.  Taking~$\mu$ to be surface  measure~$\sigma_{r}$ on the sphere of radius $r$,  the inequality simplifies to
\begin{equation}\label{sphere}
\sum_{j=1}^N \int  |f_j(x)|^2 d\sigma_{r}(x) \le C_d\,r^{\frac{d-1}{d+1}}\bigg(\sum_{j=1}^N \|\nabla f_j\|^2_{L^2(\R^d)}\bigg)^{\frac{d}{d+1}}
\end{equation}
whenever $(f_j)_{j=1,\ldots,N}$ is orthonormal in $L^2(\R^d)$, $N\ge1$ and $r>0$. Covering the sphere with $L^2(\R^d)$-normalised bump functions $f_j$, with small disjoint supports, it is easy to calculate that the exponent of the sum is optimal. Covering with a single bump function, supported on $B(0,2r)$, the exponent of $r$ is also smallest possible.

When $\mu$ is absolutely continuous with respect to Lebesgue measure,  inequalities like these have proven useful in Mathematical Physics, 
 where the $(f_j)_j$ are wavefunctions of electrons and the orthonormality corresponds to the Pauli Exclusion Principle. With $d\mu(x)=V(x)dx$, 
the CLR inequality \cites{C, L, R} provides a bound for the number of negative eigenvalues of $-\Delta+V$ (or number of bound electrons), and the Lieb--Thirring inequality \cite{LT} provides a bound for the sum of negative eigenvalues (preventing the electrons from getting too close to the nuclei). These inequalities can be combined to prove the stability of matter; see for example~\cites{BL0, LS}.
 
 By multiplying \eqref{sphere} by $V(r)$, and integrating with respect to $r$, one obtains an estimate for the potential energy, with radial potential, in terms of its kinetic energy. However the Frostman condition amounts to nonradial~$V$ belonging to a Morrey space with norm defined by
$$
\|V\|_{\alpha}:=\sup_{x\in\mathbb{R}^d,\,r>0}\frac{1}{|B(x,r)|^{\alpha/d}}\int_{B(x,r)}|V(y)|\,dy.
$$
Taking $s_1=s$, $s_2=0$ and $\alpha=d-1$ in Theorem~\ref{1.2} yields the following result.
 
 \begin{cor}\label{cor0.1} Let $0<\theta=\frac{d-2s}{d-1}<1$. Then there is a constant $C>0$ such that
$$
\sum_{j=1}^N \int_{\R^d}|D^{-s}\psi_j(x)|^2V(x)\,dx \le  C\|V\|_{d-1}^{\theta}\|V\|^{1-\theta}_{L^{1}(\R^d)}N^{\theta}
$$
whenever $(\psi_j)_{j=1,\ldots,N}$ is orthonormal in $L^2(\R^d)$, $N\ge 1$ and $V:\R^d\to \R$.
\end{cor}

When the potential consists of a sum of Coulomb potentials, with singularities of  the form $|\cdot|^{-1}$, the Morrey norm on the right-hand side is finite. However, the CLR and Lieb--Thirring inequities (which provide estimates for potentials in  $L^{p}(\R^d)$ with $p>1$) are more suited to this kind of scenario. On the other hand, our estimate can also deal with shell potentials, like in the spherical example above. Denoting the Euclidean distance to the unit sphere by $d(\cdot,S^{d-1})$, a potential of the form $\mathrm{1}_{B(0,2)}d(\cdot, S^{d-1})^{-\lambda}$, with $\lambda<1$, does not belong to the previously mentioned $L^{p}$-spaces, however it belongs to both $L^1(\R^d)$ and the $(d-1)$--Morrey space. 

From this point of view, the spherical trace inequality \eqref{sphere} can be interpreted as a bound for the potential energy by the kinetic energy, with shell potential~$\sigma_{r}$.
Looking back further, Theorem~\ref{1.2} can be viewed as a bound for the potential~$\mu$. This complements the work of Frank \cite{F2} and Frank and Laptev \cite{FL} who proved stronger estimates for potentials supported on hyperplanes. In the final section, we recover the Rozenblum--Tashchiyan bound~\cite{RT} for the the number of negative eigenvalues of $(-\Delta)^{s}-\mu$, as well as Rozenblum's bound \cite{R1} for the sum of  negative eigenvalues. This final implication can be reversed to obtain Theorem~\ref{1.1} whenever  $s_1=0$ and $s_2\in \N$. In particular, the spherical trace estimate \eqref{sphere} can be deduced from Rozenblum's eigenvalue bound via a variational argument.

For the proof of Theorem~\ref{1.1}, we will consider a kind of duality in order to employ analytic interpolation, however the arguments will be otherwise direct. We avoid  Schatten classes and variational arguments, instead relying on an application of Bessel's inequality; see for example \cites{BHLNS,N0}. The dual operator $(D^{-s})^*$ is similar to~$D^{-s}$, but with the roles of Lebesgue measure and~$\mu$ reversed, so  we will work in a generalised framework with two Frostman measures. 

Throughout, the nonnegative Borel measure $\nu$ will satisfy $$
[\nu]_\beta:=\sup_{x\in\mathbb{R}^d,\,r>0}\frac{\nu\big(B(x,r)\big)}{r^{\beta}}<\infty,\qquad
0<\beta\leq d.
$$
A consequence of our arguments will be the following estimate for the interaction energy. This complements the Cauchy--Schwarz bound by the product of the individual energies; see for example \cite[Lemma~3.1]{By} or \cite[Chapter~1]{La}. The Cauchy--Schwarz bound is only effective when $\alpha=\beta$, whereas the following estimate cannot hold when $\alpha=\beta$, as that would include a bound for the (often infinite) individual $\alpha$-energy of an $\alpha$-Frostman measure.

\begin{cor}\label{cor1.1} Let $0<\alpha\neq\beta\le d$. Then there is a constant $C>0$ such that
\begin{align*}
        \int_{E} \int_{F}\frac{d\nu(y)d\mu(x)}{|x-y|^{\frac{\alpha+\beta}{2}}}\le C\Big(\mu(E)[\mu]_\alpha\Big)^{1/2}\Big(\nu(F)[\nu]_\beta\Big)^{1/2}.
    \end{align*}
   \end{cor}
In the Fourier analytic proof, we will employ a  technique, previously used by Bourgain~\cite{B}, which takes advantage of a separation of high and low frequencies within the Marcinkiewicz real interpolation argument.

\section{Notation}\label{decomp}

From now on we will write $A\ls B$ if there is a constant $C>0$, depending only on $d,\alpha,\beta, p, q,r, s$  such that $A\le CB$. Moreover, by dividing $\mu$ by $[\mu]_\alpha$ and $\nu$ by~$[\nu]_\beta$, there is no loss of generality in supposing that $[\mu]_\alpha=[\nu]_\beta=1$. Therefore, we will only explicitly write the Frostman dependence of the constants when it is convenient and the Frostman normalisation should be assumed otherwise.

Let smooth  $\varphi:\R\to[0,1]$ be supported in $(1/2,2)$. We write $D^{-s}_{R}$ for the fractional derivative combined with the Littlewood--Paley projection;
$$
 D^{-s}_{R}[f\nu]:=\Big(\,|\cdot|^{-s}\varphi(R^{-1}|\cdot|)\widehat{f\nu}\,\Big)^\vee.
$$
The Fourier transform and its inverse are formally defined by
$$
\widehat{f\nu}:=\frac{1}{(2\pi)^{d/2}}\int_{\R^d} e^{-i\langle x, \cdot\rangle}f(x)\,d\nu(x),\quad \quad g^\vee:=\frac{1}{(2\pi)^{d/2}}\int_{\R^d} e^{i\langle \cdot, \xi\rangle}g(\xi)\,d\xi.
$$
Let smooth and nonincreasing $\phi:[0,\infty)\to[0,1]$ be equal to 1 on $[0,1]$ and supported in $[0,2)$. Then $\varphi:=\phi-\phi(2\,\cdot)$ is nonnegative and satisfies $\sum_{k\in\mathbb{Z}}\varphi(2^{-k}\cdot)=1$, so that
 $D^{-s}=D^{-s}_{\le R}+D^{-s}_{> R}$, where
$$
D^{-s}_{\le R}:=\sum_{2^{k}\le R} D^{-s}_{2^k},\qquad D^{-s}_{> R}:=\sum_{2^{k}> R}D^{-s}_{2^k}.
$$
The constants will also mildly depend on this $\varphi$ which will not change throughout.

\section{Proof of Theorem~\ref{1.2}}

We will need the following version of Theorem~A at the $p=q$ endpoint. As usual,  the H\"older conjugate is defined via the relation $1/p+1/p'=1$.
\begin{lemma}\label{lem2.1} Let $1\le p\le q\le \infty$. Then
\begin{equation*}
\big\|D^{-s}_{R} [f\nu]\big\|_{L^q(d\mu)}\ls [\mu]^{1/q}_\alpha[\nu]^{1/p'}_\beta
R^{d-\beta/p'-\alpha/q-s}\|f\|_{L^p(d\nu)}
\end{equation*}
whenever $f\in L^p(d\nu)$.
\end{lemma}

\begin{proof} 
We use the well-known decay of the Fourier transform of compactly supported smooth functions to obtain
$$
\Big|\Big(|\cdot|^{-s}\varphi(|\cdot|)\Big)^\vee\Big|\ls \frac{1}{(1+|\cdot|)^{d+1}}.
$$
The inverse Fourier transform of a product can be written as a convolution, so by rescaling we find
\begin{align*}
\big|D^{-s}_{R} [f\nu](x)\big|&\ls R^{d-s}\int \frac{|f(y)|\,d\nu(y)}{(1+|R(x-y)|)^{d+1}}\\
&\ls R^{d-s}\sum_{k\ge0} \int_{B(x,2^kR^{-1})}2^{-(d+1)k} |f(y)|\,d\nu(y).
\end{align*}
Using the fact that $\nu\big(B(x,2^{k}R^{-1})\big)\le [\nu]_\beta2^{\beta k}R^{-\beta}$, by two applications of H\"older's inequality, this yields
\begin{align}\nonumber
\big|D^{-s}_{R} [f\nu](x)\big|&\ls   R^{d-s-\beta/p'}\sum_{k\ge0}2^{-(d+1)k+\beta k/p'}\|f\|_{L^p(d\nu(B(x,2^{k}R^{-1})))}\\
&\ls   R^{d-s-\beta/p'}\Big(\sum_{k\ge0}2^{-(d+1)k}\|f\|^p_{L^p(d\nu(B(x,2^{k}R^{-1})))}\Big)^{1/p},\label{before}
\end{align}
where in the second application we have summed a geometric series.

Now, considering a partition of $\R^d$ into cubes $Q\in \mathcal{Q}$ of side $2(\sqrt{d}R)^{-1}$, on each cube we have 
\begin{align*}
\big\|D^{-s}_{R} [f\nu]\big\|_{L^q(d\mu(Q))}\le [\mu]^{1/q}_\alpha R^{-\alpha/q}\sup_{x\in Q}|D^{-s}_{R} [f\nu](x)|.
\end{align*}
Here we have used that  $\mu(Q)\le [\mu]_\alpha R^{-\alpha}$.
Combining with \eqref{before}, we find that
\begin{align*}
\big\|D^{-s}_{R} [f\nu]&\big\|_{L^q(d\mu(Q))} \ls R^{d-s-\beta/p'-\alpha/q}\Big(\sum_{k\ge0}2^{-(d+1)k}\|f\|^p_{L^p(d\nu(B(x_Q,2^{k+1}R^{-1})))}\Big)^{1/p},
\end{align*}
where $x_Q$ denotes the centre of $Q$. Given that $\ell^p$ is embedded in $\ell^q$, this yields
\begin{align*}
\big\|D^{-s}_{R} [f\nu]\big\|_{L^q(d\mu)}&= \Big(\sum_{Q\in \mathcal{Q}}\big\|D^{-s}_{R} [f\nu]\big\|^q_{L^q(d\mu(Q))}\Big)^{1/q}\\
&\ls R^{d-s-\beta/p'-\alpha/q}\Big(\sum_{Q\in \mathcal{Q}}\sum_{k\ge0}2^{-(d+1)k}\|f\|^p_{L^p(d\nu(B(x_Q,2^{k+1}R^{-1})))}\Big)^{1/p}.
\end{align*}
Finally, we can change the order of the sums, by Fubini's theorem, and use the finite overlapping of the balls, incurring a factor of $d^{d/2}2^{d(k+1)}$, before summing another geometric series to obtain the desired inequality.
\end{proof}

In the proof of the following theorem we adapt an argument due to Nam \cite[Theorem 2.2]{N0}, in which he simultaneously proved Lieb's Hardy--Littlewood--Sobolev inequality \cites{L} and the Lieb--Thirring kinetic inequality \cite{LT}. The Frostman measure prevents us from using Plancherel's identity and the momentum decomposition, however we get around this by using a real interpolation trick previously employed by Bourgain~\cite{B}.

\begin{theorem}\label{3.1} Let $0<2/q=\frac{d-2s_1}{\alpha+2s_2}<1$. Then there is a constant $C>0$ such that
$$
\Bigg\|\sum_{j=1}^N |D^{-s_1}f_j|^2\Bigg\|_{L^{q/2,\infty}(d\mu)} \le  C[\mu]_\alpha^{2/q}\bigg(\sum_{j=1}^N \|D^{s_2}f_j\|^2_{L^2(\R^d)}\bigg)^{2/q}
$$
whenever $(f_j)_{j=1,\ldots,N}$ is orthonormal in $L^2(\R^d)$ and $N\ge 1$.
\end{theorem}

\begin{proof} 
Recalling the frequency decomposition of the previous section, by multiplying out the square and Young's inequality,
\begin{align}\label{recall}
\frac{1}{2}\sum_{j=1}^N |D^{-s_1}f_j|^2&\le \sum_{j=1}^N |D^{-s_1}_{\le R}[f_j]\big|^2+\sum_{j=1}^N |D^{-s_1}_{> R}[f_j]\big|^2.
\end{align}
In particular, at any given point $x\in\R^d$, regardless of the choice of $R$, one of the two terms on the right-hand side must be somewhat large.

For the lower frequencies, we write
\begin{align*}
\sum_{j=1}^N \big|D^{-s_1}_{\le R}[f_j](x)\big|^2= \frac{1}{(2\pi)^{d}}\sum_{j=1}^N \Big|\Big\langle e^{i\langle x, \cdot\rangle}|\cdot|^{-s_1}\!\!\sum_{2^k\le R}\varphi(|2^{-k}\cdot|), \overline{\widehat{f}_j} \Big\rangle\Big|^2.
\end{align*}
As the $\widehat{f}_j$ are orthonormal and $s_1<d/2$, we can use Bessel's inequality to find
\begin{align*}\label{low}
\sum_{j=1}^N \big|D^{-s_1}_{\le R}[f_j](x)\big|^2&\le  \frac{1}{(2\pi)^{d}}\Big\||\cdot|^{-s_1}\!\!\sum_{2^k\le R}\varphi(|2^{-k}\cdot|)\Big\|^2_{L^2(\R^d)}
\ls  R^{d-2s_1}.
\end{align*}
Taking $R^{d-2s_1}=ct$ with $c$ sufficiently small, we bound the lower frequencies by~$t/4$. Recalling \eqref{recall}, this  forces the inclusion
\begin{equation}\label{inclusion}
\Big\{ \,x\,:\,\sum_{j=1}^N \big|D^{-s_1}f_j(x)\big|^2> t\,\Big\}\subset \Big\{  \,x\,:\,\sum_{j=1}^N \big|D^{-s_1}_{> R}[f_j](x)\big|^2> t/4\,\Big\}.
\end{equation}

For the higher frequencies, the $\ell^1$-triangle inequality followed by the $L^2(d\mu)$-triangle inequality yields
$$
\bigg\|\sum_{j=1}^N\big|D^{-s_1}_{> R}[f_j]\big|^2\bigg\|_{L^{1}(d\mu)}\le \sum_{j=1}^N\Big(\sum_{2^k> R} \big\|D^{-s_1}_{2^k}[f_j]\big\|_{L^2(d\mu)}\Big)^2.
$$
Using the $p=q=2$ version of Lemma \ref{lem2.1}, with $s=s_1+s_2$ and $\beta=d$, we find
\begin{equation*}\label{forlat}
\bigg\|\sum_{j=1}^N\big|D^{-s_1}_{> R}[f_j]\big|^2\bigg\|_{L^{1}(d\mu)}
\ls  R^{d-\alpha-2(s_1+s_2)}\sum_{j=1}^N\big\|D^{s_2}f_j\big\|^2_{L^{2}(\R^d)}.
\end{equation*}
Here, as $2(s_1+s_2)>d-\alpha$,  we are able to sum the resulting geometric sequence.
Then, by an application of Tchebyshev's inequality, this yields
\begin{equation*}
\mu\Big(\Big\{\,x\, :\, \sum_{j=1}^N \big|D^{-s_1}_{> R}[f_j]\big(x)|^2> t/4\,\Big\}\Big)\ls  t^{-1}R^{d-\alpha-2(s_1+s_2)}\sum_{j=1}^N\big\|D^{s_2}f_j\big\|^2_{L^{2}(\R^d)}.
\end{equation*}
Taking $R^{d-2s_1}=ct$ and combining with \eqref{inclusion}, we find
$$
\mu\bigg(\Big\{\,x\in \mathbb{R}^d\, :\, \sum_{j=1}^N \big|D^{-s_1}f_j(x)\big|^2> t\,\Big\}\bigg)\ls  t^{-\frac{\alpha+2s_2}{d-2s_1}}\sum_{j=1}^N\big\|D^{s_2}f_j\big\|^2_{L^{2}(\R^d)},
$$
which is the desired weak-type estimate.
\end{proof}

By an application of the Lorentz--H\"older inequality (see for example \cite[Section~1.4]{G}), we obtain the following more general version of Theorem~\ref{1.2} (where~$g$ was taken to be the indicator function $\mathrm{1}_\Gamma$ of $\Gamma$). Here~$\widetilde{q}$ satisfies $2/q+2/\widetilde{q}=1$. 

\begin{theorem}\label{3.2} Let $0<2/q=\frac{d-2s_1}{\alpha+2s_2}<1$. Then there is a $C>0$ such that
$$
\sum_{j=1}^N \big\|gD^{-s_1}f_j\big\|^2_{L^{2}(d\mu)} \le  C[\mu]_\alpha^{2/q}\|g\|^2_{L^{\widetilde{q},2}(d\mu)}\bigg(\sum_{j=1}^N \|D^{s_2}f_j\|^2_{L^2(\R^d)}\bigg)^{2/q}
$$
whenever $(f_j)_{j=1,\ldots,N}$ is orthonormal in $L^2(\R^d)$ and $N\ge 1$.
\end{theorem}

\section{Lorentz refinements and interaction energies}

For the proof of Theorem~\ref{1.1}, we will also require the following Lorentz-refined version of Theorem~A. We provide a proof based on Lemma~\ref{lem2.1} from which we are able to deduce the interaction energy estimate; see also~\cites{GG,GM} for the $\mu=\nu$ case.

 \begin{theorem}\label{2.1} Let $1< p< q<\infty$ and $\lambda=\alpha/q+\beta/p'$. Then
  \begin{align*}
        \left|\iint\frac{f(y)g(x)}{|y-x|^{\lambda}}d\nu(y)d\mu(x)\right|\lesssim [\mu]_\alpha^{1/q}[\nu]_\beta^{1/p'}\|f\|_{L^{p,2}(d\nu)}\|g\|_{L^{q'\!,2}(d\mu)}.
    \end{align*}
\end{theorem}

By considering $d\mu(x)=|x|^{-(d-\alpha)}dx$ and  $d\nu(y)=|y|^{-(d-\beta)}dy$, many cases of the Stein--Weiss inequality \cite{SW0} can be recovered. Via this connection it is already possible to see that $p\le q$ is necessary when $\alpha<\beta$  (see \cite[Section~2.3]{N}), however $p<q$ is also necessary in our generality, by considering~$\mu$ to be $(d-1)$-dimensional Lebesgue measure (see for example \cite[pp. 208-209]{KS}). 

First we  consider a weaker substitute for the $p=q$ endpoint case. When one of the two measures, say $\nu$, is Lebesgue measure, the following estimate is a consequence of a restricted strong-type $(p,p)$ estimate  due to Korobkov and Kristensen \cite[Theorem 1.3]{KK}.
More recently, Mihula, Pick and Spector generalised their result so that $\nu$ is a Frostman measure (see \cite[Theorem~1.6]{MPS}), however at the expense of~$\mu$ depending on their $\nu$, at all scales, reminiscent of the classical two-weight Muckenhoupt condition; see for example \cite{PSSY}. Here we treat two independent Frostman measures.  Corollary~\ref{cor1.1} of the introduction follows by taking $p=2$.

 \begin{theorem}\label{2.2} Let $1< p<\infty$ and  $\lambda=\alpha/p+\beta/p'$ with $\alpha\neq \beta$.  Then
\begin{align*}
        \int_{E} \int_{F}\frac{d\nu(y)d\mu(x)}{|y-x|^{\lambda}}\ls [\mu]_\alpha^{1/p}[\nu]_\beta^{1/p'}\mu(E)^{1/p'}\nu(F)^{1/p}.
    \end{align*}
\end{theorem}

\begin{proof} By symmetry we can suppose  that $\alpha<\beta$. Then, by an application of the Lorentz--H\"older inequality, we have
\begin{align*}
        \int_{E} \int_{F}\frac{d\nu(y)d\mu(x)}{|y-x|^{\lambda}}\ls \left\|\int_{F} \frac{d\nu(y)}{|y-\cdot\,|^{\lambda}}\right\|_{L^{p,\infty}(d\mu)}\mu(E)^{1/p'}.
        \end{align*}
Therefore, as the Fourier transform of $|\cdot|^{-\lambda}$ is a constant multiple of $|\cdot|^{\lambda-d}$  (see for example \cite[Theorem 2.4.6]{G}), it will suffice to prove 
 \begin{equation*}\label{name}
 \big\|D^{\lambda-d}[\nu\mathrm{1}_F]\big\|_{L^{p,\infty}(d\mu)}
\ls  \nu(F)^{1/p}.
\end{equation*} 
Letting  $p_0<p<p_1$ satisfy $\frac{1}{2p_0}+\frac{1}{2p_1}=\frac{1}{p}$,  by two different versions of Lemma~\ref{lem2.1}, we obtain
\begin{equation*}
\big\|D^{\lambda-d}_{\le R} [\nu\mathrm{1}_F]\big\|_{L^{p_0}(d\mu)}\ls
R^{\varepsilon}\nu(F)^{1/p_0}
\end{equation*}
 on the one hand, where $\varepsilon=(\beta-\alpha)(1/p_0-1/p)$, and 
\begin{equation*}
\big\|D^{\lambda-d}_{> R} [\nu\mathrm{1}_F]\big\|_{L^{p_1}(d\mu)}\ls
R^{-\varepsilon}\nu(F)^{1/p_1}
\end{equation*}
on the other. These follow using the triangle inequality and summing geometric series as before. By Tchebyshev's inequality, we find that
 $$
\mu\Big(\big\{\, x\in \mathbb{R}^d\,:\, \big|D^{\lambda-d}_{\le R}[\nu\mathrm{1}_F](x)\big|> t/2\,\big\}\Big) \ls  t^{-p_0}R^{p_0\varepsilon}\nu(F)
$$
and
$$
\mu\Big(\big\{\, x\in \mathbb{R}^d\,:\, \big|D^{\lambda-d}_{> R}[\nu\mathrm{1}_F](x)\big|> t/2\,\big\}\Big) \ls  t^{-p_1}R^{-p_1\varepsilon}\nu(F).
$$
Taking $R^\varepsilon=t^{\frac{p_0-p_1}{p_0+p_1}}$, the right-hand sides of these bounds are both $t^{-p}\nu(F)$.
So if
$$
 \big|D^{\lambda-d}_{\le R}[\nu\mathrm{1}_F]\big|+\big|D^{\lambda-d}_{> R}[\nu\mathrm{1}_F]\big|\ge \big|D^{\lambda-d}[\nu\mathrm{1}_F]\big|> t,
$$
then one of the two left-hand terms is forced to be larger than $t/2$, yielding
$$
 \mu\left(\left\{\, x\in \mathbb{R}^d\,:\, \big|D^{\lambda-d}[\nu\mathrm{1}_F](x)\big| > t\,\right\}\right)
\ls  t^{-p}\nu(F)
$$
which completes the proof.
\end{proof}

By similar arguments, we can now complete the proof of Theorem~\ref{2.1}. By rewriting the convolution as a fractional derivative and applying the Lorentz--H\"older inequality as before, it is a consequence of the following estimate.

\begin{theorem}\label{forlater} Let $1< p< q<\infty$ and $\lambda=\alpha/q+\beta/p'$. Then
\begin{equation*}
\left\|D^{\lambda-d}[f\nu]\right\|_{L^{q,2}(d\mu)}\ls [\mu]_\alpha^{1/q}[\nu]_\beta^{1/p'}\|f\|_{L^{p,2}(d\nu)}.
\end{equation*}
\end{theorem}

\begin{proof} 
We use two versions of  Lemma~\ref{lem2.1}, with $\lambda-d=-s$, to obtain
\begin{equation}\label{sup}
\big\|D^{\lambda-d}_{\le  R} [f\nu]\big\|_{L^{\infty}} \ls R^{\alpha/q}\|f\|_{L^p(d\nu)}
\end{equation}
on the one hand, and
\begin{equation*}\label{p}
\big\|D^{\lambda-d}_{>  R} [f\nu]\big\|_{L^p(d\mu)} \ls R^{-\alpha(1/p-1/q)}\|f\|_{L^p(d\nu)}
\end{equation*} 
on the other. Here we used that $p<q$ in order to sum the resulting series.
Then, by Tchebyshev's inequality, this implies that
 \begin{equation}\label{popl}
 \mu \left(\left\{\, x\in \, \mathbb{R}^d :\, \big|D^{\lambda-d}_{> R}[f\nu](x)\big|> t/2\,\right\}\right) \ls  t^{-p}R^{-\alpha(1-p/q)}\|f\|^p_{L^p(d\nu)}.
 \end{equation}
By \eqref{sup}, we can take $R=(ct)^{q/\alpha}\|f\|^{-q/\alpha}_{L^p(d\nu)}$ with $c$ small enough so that the lower frequencies are bounded by $t/2$, yielding
 $$
 \left\{\, x\,:\, \big|D^{\lambda-d}[f\nu](x)\big|> t\,\right\}\subset \left\{\, x\, : \big|D^{\lambda-d}_{>  R}[f\nu](x)\big|>t/2\,\right\}.
 $$
 Then, taking our value for $R$ in \eqref{popl}, we find that
\begin{equation}\label{weak}
 \mu\left(\left\{\, x\in \mathbb{R}^d\,:\, \big|D^{\lambda-d}[f\nu](x)\big|> t\,\right\}\right)
\ls  t^{-q}\|f\|^q_{L^{p}(d\nu)},
\end{equation}
which is a weak-type version of the desired estimate.

In order to upgrade to a strong-type $(p,q)$ estimate, we consider a version of ~\eqref{weak} with a $p_0$ smaller than $p$ (which fixes a $q_0$  smaller than $q$ via the relation $\lambda=\alpha/q_0+\beta/p_0'$) and a different version of~\eqref{weak} with a $p_1$ larger than $p$ (which similarly fixes a $q_1$ larger than~$q$). Then by the Lorentz-refined version of the Marcinkiewicz interpolation theorem (see for example \cite[Section 5.3]{BL} or \cite[Theorem 1.4.19]{G}), we can interpolate between the two estimates to complete the proof.
\end{proof}

\section{Proof of Theorem~\ref{1.1}}

Our orthonormal functions are now fixed in $L^2(\R^d)$ providing less flexibility. In order to upgrade Theorem~\ref{3.1} to a strong-type estimate, we use Stein's analytic interpolation  so that $s$ can vary accordingly with the Lebesgue exponents; see~\cite[Section~5.4]{SW}. 
For $z\in \mathbb{C}$, we consider $D^{z}$  formally defined as before by
$$
D^{z}f:=\big(|\cdot|^{z}\widehat{f}\,\big)^\vee.
$$ 
We also need to pass to a kind of dual estimate in order to consider a linear operator, as in \cite[Section 1.3]{F}, however we stop short of appealing to Schatten classes. 

With $r=q$, our desired estimate takes the form
\begin{align*}
\bigg\|\sum_{j} \gamma_j|D^{z}f_j|^2\bigg\|_{L^{q/2,r/2}(d\mu)}
\ls  \|\gamma\|_{\ell^{p/2}},
\end{align*}
which, by Lorentz duality (see for example \cite[Section 1.4]{G}), would follow from
$$
\int_{\R^d} \sum_{j}|\gamma_j| |D^{z}f_j(x)|^2|g(x)|^2d\mu(x) \ls \|\gamma\|_{\ell^{p/2}}\|g\|^2_{L^{\widetilde{q},\widetilde{r}}(d\mu)}.
$$
Recall the half conjugates satisfy $2/p+2/\widetilde{p}=1$.
Then, by Fubini's theorem and an application of H\"older's inequality, this in turn would follow  from 
$$
\bigg(\sum_{j}\|gD^{z}f_j\|^{\widetilde{p}}_{L^2(d\mu)}\bigg)^{1/\widetilde{p}} \ls  \|g\|_{L^{\widetilde{q},\widetilde{r}}(d\mu)}.
$$
Indeed, the three estimates are equivalent, so  Theorem~\ref{1.1} is a consequence of the following estimate which we are now in a position to prove.
\begin{theorem}\label{lem3.1} Let $2<\widetilde{q}<\widetilde{p}$ and $\Re(z)=-(d/2-\alpha/q)$. Then
\begin{equation}\label{desired}
        \bigg(\sum_{j}\|gD^{z}f_j\|_{L^2(d\mu)}^{\widetilde{p}}\bigg)^{1/\widetilde{p}}\lesssim \|g\|_{L^{\widetilde{q}}(d\mu)}
    \end{equation}
    whenever $(f_j)_j$ is orthonormal in $L^2(\R^d)$.
\end{theorem}

\begin{proof} By Theorem~\ref{forlater} with $p=2$, for $s=d-\lambda=d/2-\alpha/q$,  we have
\begin{equation*}
\|D^{-s}f_j\|_{L^{q,2}(d\mu)}
\ls \|f_j\|_{L^2(\R^d)}.
\end{equation*}
Noting that $|\xi|^{z}=|\xi|^{\Re(z)}e^{i\theta}$, the unimodular factor can be passed to $\widehat{f}_j(\xi)$ without changing the $L^2$-norm, so this estimate continues to hold with $D^{-s}$ replaced by $D^{z}$ as long as $\Re(z)=-(d/2-\alpha/q)$.
Then, by an application of the Lorentz--H\"older inequality, we obtain the following lemma, which is slightly better than what we need when $\widetilde{q}$ is very large.
\begin{lemma}\label{lem3.5} Let $2<\widetilde{q}<\infty$ and let $\mathrm{Re}(z)=-(d/2-\alpha/q)$. Then
\begin{align*}
         \sup_{j}\|gD^{z}f_j\|_{L^2(d\mu)} \lesssim \|g\|_{L^{\widetilde{q}, \infty}(d\mu)}
    \end{align*}
    whenever $(f_j)_j$ is orthonormal in $L^2(\R^d)$.
\end{lemma}

On the other hand, as $\ell^{p/2}$ is embedded in $\ell^{q/2,1}$,  the following Lorentz-weakened version of what we need is a consequence of Theorem~\ref{3.1} with $s_2=0$.

\begin{lemma}\label{lem3.2} Let $2<\widetilde{q}<\widetilde{p}$ and let $\mathrm{Re}(z)=-(d/2-\alpha/q)$. Then
\begin{align*}
        \bigg(\sum_{j}\|gD^{z}f_j\|_{L^2(d\mu)}^{\widetilde{p}}\bigg)^{1/\widetilde{p}}\lesssim\|g\|_{L^{\widetilde{q},2}(d\mu)}
    \end{align*}
    whenever $(f_j)_j$ is orthonormal in $L^2(\R^d)$.
\end{lemma}

 In order to see that we can obtain the desired estimate \eqref{desired} by interpolation between  Lemmas~\ref{lem3.1} and \ref{lem3.2}, we first note that the $\ell^{\widetilde{p}}L^{2}(d\mu)$-spaces interpolate with $\widetilde{p}$ following the usual convexity rules (see for example \cite[Theorem 5.1.2]{BL}). On the other hand, we also need to check that the resulting Lorentz exponent $\widetilde{r}$ is as large as the resulting Lebesgue exponent $\widetilde{q}$. Note that the Lorentz exponents follow the same convexity rules as the Lebesgue exponents when interpolating using the complex method; see for example ~\cite{HS}.

We need to prove \eqref{desired} with~$\widetilde{p}$ arbitrarily close to $\widetilde{q}$, so we interpolate with a Lebesgue exponent $\widetilde{q}_1$ from Lemma~\ref{lem3.1} that is increasingly large, and additionally require that~$\widetilde{r}=\widetilde{q}$. In that case, we have 
$$
\frac{1}{\widetilde{r}}=\frac{\theta}{2}= \frac{\theta}{\widetilde{q}_0}+\frac{1-\theta}{\widetilde{q}_1}=\frac{1}{\widetilde{q}},
$$
where $\theta$ is the interpolation parameter and $\widetilde{q}_0$ is the Lebesgue exponent from Lemma~\ref{lem3.2}. We see that $\theta$ must equal $2/\widetilde{q}$ and so, making this substitution, we find that $\widetilde{q}_0$ must satisfy the relation
$$
\frac{2}{\widetilde{q}_0}+\frac{\widetilde{q}-2}{\widetilde{q}_1}=1.
$$
This simply forces our choice of $\widetilde{q}_0$ closer to~2, as $\widetilde{q}_1$ gets larger, as is permitted.

It remains to show that $T_z: g\to (gD^zf_j)_j$ forms a family of analytic linear operators on any closed strip contained in 
$$\mathcal{S}:=\Big\{\,z\in\mathbb{C}\,:\, -\frac{d}{2}<\mathrm{Re}(z)<- \frac{d-\alpha}{2}\,\Big\};$$
see for example \cite[Section 5.4]{SW}.
By this we mean that, for any choice of simple $\gamma$, $g$ and~$h$, the function
\begin{align*}
    F(z):= \sum_j \gamma_j\int_{\R^d}  g(x)D^zf_j(x) h(x)\,d\mu(x)
    \end{align*}
is analytic on $\mathcal{S}$ and uniformly bounded on any closed strip within it. 

For the analyticity, we consider $z\in \mathcal{S}$ and  $\varepsilon>0$ such that $B(z,2\varepsilon)\subset \mathcal{S}$. 
Note that by an application of Plancherel's identity, we can also write
$$
F(z)= \sum_j \gamma_j\int_{\R^d}  |\xi|^z\widehat{gh\mu}(\xi)\widehat{f}_j(-\xi)\,d\xi.
$$
Since~$|\xi|^{z}=e^{z\ln |\xi|}$ is analytic as a map $\mathcal{S}\to\mathbb{C}$, by an application of the complex mean value inequality, we find that
\begin{align*}
   \sup_{w\in B(z,\varepsilon/2)} \left|\frac{|\xi|^{w}-|\xi|^{z}}{w-z}\right|&\ls \big|\ln|\xi|\big|\sup_{w\in B(z,\varepsilon/2)}|\xi|^{\Re(w)}\\
   &\ls |\xi|^{\Re(z)}\big(|\xi|^\varepsilon+|\xi|^{-\varepsilon}\big)=:b(\xi).
    \end{align*}
The second term in the sum is so that the bound also holds when $0<|\xi|<1$. By the dominated convergence theorem, it will suffice prove that
\begin{align*}
     \sum_j|\gamma_j|\int  b(\xi)|\widehat{gh\mu}(\xi)\widehat{f}_j(-\xi)|\,d\xi<\infty.
\end{align*}
By an application of the Cauchy--Schwarz inequality, we can then apply Theorem~\ref{forlater} with $q=2$ (and the roles of $\mu$ and $\nu$ interchanged) to find  $p_1,p_2\in (1,2)$, such that 
\begin{align*}
  \sum_j|\gamma_j|\int  b(\xi)|\widehat{gh\mu}(\xi)\widehat{f}_j(-\xi)|\,d\xi
 & \ls  \sum_j|\gamma_j|\,\|f_j\|_{L^2(\R^d)}\|b\widehat{gh\mu}\|_{L^2(\R^d)}\\
 &\ls \|\gamma\|_{\ell^1}\Big(\|gh\|_{L^{p_1}(d\mu)}+\|gh\|_{L^{p_2}(d\mu)}\Big)<\infty,
\end{align*}
where $\Re(z)+\varepsilon=-(d/2-\alpha/p_1')$ and $\Re(z)-\varepsilon=-(d/2-\alpha/p_2')$. Thus the analyticity of $F(z)$ follows from that of $|\xi|^z$ by the dominated convergence theorem.

For the uniform boundedness, by the Cauchy--Schwarz inequality and Theorem~\ref{forlater} as before, we find that \begin{align*}
|F(z)|& \le      \|\gamma\|_{\ell^1} \big\|(D^{\overline{z}})^*[gh]\big\|_{L^2(\R^d)} \ls \|\gamma\|_{\ell^1}\|gh\|_{L^p(d\mu)},
\end{align*}
where $\Re(z)=-(d/2-\alpha/p')$. Note that Theorem~\ref{forlater} also applies to complex $(D^{\overline{z}})^*$ as the convolution with $|\cdot|^{-(d+z)}$ is at its largest when the kernel and function are both real and nonnegative. Thus $p$ does not change as $z$ moves up and down the strip. As long as $z$ stays away from the edges of $\mathcal{S}$ (where Theorem~\ref{forlater} fails to provide a bound), we can take a supremum over the closed strip to get a uniform bound. This completes the proof.
\end{proof}

\section{Applications to eigenvalue bounds}

Taking  $\mu$ to be $d$-dimensional Lebesgue measure and $g=\sqrt{V}$ in Theorem~\ref{3.2}, we recover nonendpoint versions of  the kinetic Lieb--Thirring inequality \cite{LT} and the CLR inequality \cites{C,L, R}.
However, we have gained some flexibility, and we can make other choices depending on the nature of our potential. Taking $g=1$ and $d\mu(x)=V(x)dx$, we obtain the following Morrey space variation.

\begin{theorem}\label{5.1} Let $0<2/q=\frac{d-2s_1}{\alpha+2s_2}<1$. Then there is a constant $C>0$ such that
$$
\sum_{j=1}^N \int|D^{-s_1}\psi_j(x)|^2V(x)\,dx \le  C\|V\|_{\alpha}^{2/q}\|V\|^{1-2/q}_{L^{1}(\R^d)}\bigg(\sum_{j=1}^N \|D^{s_2}\psi_j\|^2_{L^2(\R^d)}\bigg)^{2/q}
$$
whenever $(\psi_j)_{j=1,\ldots,N}$ is orthonormal in $L^2(\R^d)$, $N\ge 1$ and $V:\R^d\to \R$.
\end{theorem}

Taking $s_1=0$ and $s_2=s$, we obtain a bound for the potential energy in terms of the fractional kinetic energy. 
\begin{cor}\label{cor5.2} Let $0< d-2s<\alpha\le d$. Then there is a constant $C>0$ such that
$$
\sum_{j=1}^N \int|\psi_j(x)|^2V(x)\,dx \le C\|V\|_{\alpha}^{\frac{d}{\alpha+2s}}\|V\|^{1-\frac{d}{\alpha+2s}}_{L^{1}(\R^d)}\bigg(\sum_{j=1}^N \|D^s \psi_j\|^2_{L^2(\R^d)}\bigg)^{\frac{d}{\alpha+2s}}
$$
whenever $(\psi_j)_{j=1,\ldots,N}$ is orthonormal in $L^2(\R^d)$, $N\ge 1$ and $V:\R^d\to \R$.
\end{cor}

Letting $(E_j)_{j=1,\dots,N}$ denote negative eigenvalues of $(-\Delta)^{s}+V$ with $L^2(\R^d)$-normalised eigenfunctions $(\psi_j)_{j=1,\ldots,N}$, we have
$$
\sum_{j=1}^N E_j =\sum_{j=1}^N\Big\langle \psi_j,\big((-\Delta)^{s}+V\big)\psi_j\Big\rangle.
$$
Discarding the positive potential energy, this yields\begin{align*}
\sum_{j=1}^N E_j&\ge \sum_{j=1}^N \|D^s \psi_j\|^2_{L^2(\R^d)}-\sum_{j=1}^N \int|\psi_j(x)|^2V_{-}(x)\,dx,
\end{align*}
where $-V_{-}$ denotes the negative part of the potential.
As the eigenfunctions are orthogonal, we can apply Corollary~\ref{cor5.2} to find
\begin{align*}
\sum_{j=1}^N E_j
&\ge X-C\|V_{-}\|_{\alpha}^{\frac{d}{\alpha+2s}}\|V_{-}\|^{1-\frac{d}{\alpha+2s}}_{L^{1}(\R^d)}X^{\frac{d}{\alpha+2s}},\qquad X=\sum_{j=1}^N \|D^s \psi_j\|^2_{L^2(\R^d)}.
\end{align*}
 Minimising this lower bound, as a function of $X$ over all $X\ge0$, we obtain
$$
\sum_{j=1}^N E_j\gs -\|V_{-}\|_{\alpha}^{\frac{d}{\alpha+2s-d}}\|V_{-}\|_{L^{1}(\R^d)}.
$$
Recalling that the eigenvalues are negative, we conclude that
$$
\sum_{j=1}^N |E_j|\ls  \|V_{-}\|_{\alpha}^{\frac{d}{\alpha+2s-d}}\|V_{-}\|_{L^{1}(\R^d)}.
$$

By the same argument, starting with Theorem~\ref{3.2} with $g=1$, we obtain the following bound for singular potentials. When $s\in \N$, this is a particular case of a more general result of Rozenblum~\cite{R1}.
  \begin{theoremB} Let $0< d-2s<\alpha\le d$. Then there is a constant $C>0$ such that
 $$
\sum_{j=1}^N |E_j|\le  C[\mu]_{\alpha}^{\frac{d}{\alpha+2s-d}}\|\mu\|
$$
whenever $(E_j)_{j=1,\dots,N}$ are  negative eigenvalues of $(-\Delta)^{s}-\mu$.
\end{theoremB}

On the other hand, taking $s_1=s$ and $s_2=0$ in Theorem~\ref{5.1} yields the following estimate more suitable for counting the  negative eigenvalues. With $\alpha=d-1$, this is Corollary~\ref{cor0.1} of the introduction.

\begin{cor}\label{cor5.1} Let $0< d-2s<\alpha\le d$. Then there is a constant $C>0$ such that
$$
\sum_{j=1}^N \int|D^{-s}\psi_j(x)|^2V(x)\,dx \le  C\|V\|_{\alpha}^{\frac{d-2s}{\alpha}}\|V\|^{1-\frac{d-2s}{\alpha}}_{L^{1}(\R^d)}N^{\frac{d-2s}{\alpha}}
$$
whenever $(\psi_j)_{j=1,\ldots,N}$ is orthonormal in $L^2(\R^d)$, $N\ge 1$ and $V:\R^d\to \R$.
\end{cor}

To bound the number of negative eigenvalues, we use an argument due to Frank~\cite{F3}; see also \cite[page 11]{N0}. Let $W$ denote the span of the eigenfunctions $(\psi_j)_{j=1,\ldots,N}$. As $D^s$ is a strictly positive operator on $L^2(\R^d)$, we have $$\dim D^sW =N,$$ so we can find orthonormal $(D^sg_j)_{j=1,\ldots,N}$ with $g_j\in W$.
Writing the $g_j$ as linear combinations of the eigenfunctions, and calculating as before, we find
\begin{align*}
0&\ge \sum_{j=1}^N \|D^s g_j\|^2_{L^2(\R^d)}-\sum_{j=1}^N \int|g_j(x)|^2V_{-}(x)\,dx,
\end{align*}
so that
$$
N\le \sum_{j=1}^N \int  |D^{-s}f_j(x)|^2 V_{-}(x)\,dx,
$$
where the $f_j=D^sg_j$ are orthonormal. By an application of Corollary~\ref{cor5.1},  this yields
 $$
 N\ls \|V_{-}\|_{\alpha}^{\frac{d-2s}{\alpha}}\|V_{-}\|^{1-\frac{d-2s}{\alpha}}_{L^{1}(\R^d)}N^{\frac{d-2s}{\alpha}},
 $$
 and so, by rearranging, we conclude
 $$
 N\ls \|V_{-}\|_{\alpha}^{\frac{d-2s}{\alpha+2s-d}}\|V_{-}\|_{L^{1}(\R^d)}.
 $$ 
 
By the same argument, starting with Theorem~\ref{3.2} with $g=1$, we obtain the following bound for singular potentials which is a particular case of a more general result of Rozenblum and Tashchiyan~\cite{RT}.
 
 \begin{theoremC} Let $0< d-2s<\alpha\le d$. Then there is a constant $C>0$ such that
 $$
N\le  C[\mu]_{\alpha}^{\frac{d-2s}{\alpha+2s-d}}\|\mu\|
$$
whenever $(E_j)_{j=1,\dots,N}$ are  negative eigenvalues of $(-\Delta)^{s}-\mu$.
\end{theoremC}

\end{document}